\documentclass[12pt]{article}
\usepackage[utf8]{inputenc}
\usepackage[T1]{fontenc}
\usepackage{comment}
\usepackage{adjustbox}
\usepackage{mathrsfs}

\usepackage{marvosym}
\usepackage{dsfont}
\usepackage{rotating}
\usepackage{amssymb}
\usepackage{amsmath}
\usepackage{amsthm}
\usepackage{thmtools}
\usepackage[colorlinks=true, linkcolor=blue, citecolor=blue, urlcolor=blue]{hyperref}
\usepackage{cleveref}
\usepackage{verbatim}
\usepackage{enumitem}
\usepackage{url}
\usepackage{float}
\usepackage{mathtools}

\usepackage[a4paper, margin=2.5cm,top=3cm, bottom=3cm]{geometry}

\setlist[enumerate,1]{label={\textnormal{(\roman*)}}}

\theoremstyle{plain}
\newtheorem{theorem}{Theorem}[section]
\newtheorem{lemma}[theorem]{Lemma}

\newtheorem{claim}[theorem]{Claim}
\newtheorem{conjecture}[theorem]{Conjecture}
\newtheorem{observation}[theorem]{Observation}

\theoremstyle{definition}
\newtheorem{definition}[theorem]{Definition}
\newtheorem{remark}[theorem]{Remark}

\newcommand{\oldqed}{}
\def\endofClaim{\hfill\scalebox{.6}{$\Box$}}

\newenvironment{claimproof}[1][Proof]{
	\renewcommand{\oldqed}{\qedsymbol}
	\renewcommand{\qedsymbol}{\endofClaim}
	\begin{proof}[#1]
	}{
	\end{proof}
	\renewcommand{\qedsymbol}{\oldqed}
}

\newcommand{\ceil}[1]{\left\lceil#1\right\rceil}

\newcommand\N{\mathbb{N}}
\newcommand\R{\mathbb{R}}

\newcommand{\coleq}{\coloneq}

\newcommand{\Z}{\mathbb Z}

\newcommand{\A}{\mathbb A}
\newcommand{\cA}{\mathcal A}

\newcommand{\f}{\varphi}
\newcommand{\g}{\psi_{S,k}}
\newcommand{\w}{\f\vert_S}
\newcommand{\bw}{\overline \omega}
\newcommand{\bwp}[2]{\overline\omega_{#1,#2}}
\newcommand{\bwsk}{\overline\omega_{S,k}}

\newcommand{\subs}{\subseteq}
\newcommand{\sm}{\setminus}
\newcommand{\abs}[1]{\mathopen{}\left\lvert#1\right\rvert\mathclose{}}

\newcommand{\abar}{\bar{a}}
\newcommand{\bbar}{\bar{b}}
\newcommand{\pbar}{\bar{p}}
\newcommand{\q}{\mathbf{q}}

\title{No extremal square-free words over alphabets of size at least 5}
\author{Eng Keat Hng\thanks{Extremal Combinatorics and Probability Group (ECOPRO), Institute for Basic Science, Daejeon, South Korea. Email: {\tt hng@ibs.re.kr}.} \and Silas Rathke\thanks{Fachbereich Mathematik und Informatik, Freie Universität Berlin, Arnimallee 3, 14195 Berlin, Germany. Email: {\tt s.rathke@fu-berlin.de}.}}
\date{}

\hypersetup{
  pdftitle={No extremal square-free words over alphabets of size at least 5},
  pdfauthor={Eng Keat Hng and Silas Rathke},
  pdfsubject={2020 MSC: Primary 68R15; Secondary 05A05, 05-08},
  pdfkeywords={combinatorics on words, extremal words, nonchalant words}
}

\begin{document}
	
	\maketitle
	
	\begin{abstract}
		A word over an alphabet $\mathbb A$ contains a square if it has a subword of the form $XX$ where $X$ is a word. A word $W$ is \emph{extremal square-free} if it does not contain a square, but it contains a square as soon as any letter of $\mathbb A$ is inserted at any position of $W$. Grytczuk, Kordulewski, and Niewiadomski conjectured that there are no extremal square-free words over alphabets of size at least 4. We prove this for alphabets of size at least 5. Our proof also implies that the sequence of \emph{nonchalant words} defined by Grytczuk, Kordulewski, and Niewiadomski is infinite and converges to an infinite word for all alphabets of size at least 5.
	\end{abstract}


	
	\section{Introduction}
	
	An \emph{alphabet} is a finite set of symbols which are called \emph{letters}, and a \emph{word} is a finite sequence of letters drawn from an alphabet. A \emph{square} is a nonempty word of the form~$XX$, where~$X$ represents a nonempty word. A word is \emph{square-free} if it does not contain a square as a \emph{factor}, that is, a subword consisting of consecutive letters. It is easy to check that there are no binary square-free words with more than three letters. However, Thue~\cite{Thue1906} showed in 1906 that there are arbitrarily long ternary square-free words (see~\cite{Berstel}). His work~\cite{Thue1906,Thue1912} is considered to be the starting point of research in combinatorics on words~\cite{BerstelPerrin}.
	
	Recently, Grytczuk, Kordulewski, and Niewiadomski~\cite{GrytczukKordulewskiNiewiadomski} initiated the systematic study of extremal square-free words. We give a formal definition below.
	
	\begin{definition}
		Let~$W$ be a word over an alphabet~$\A$. An \emph{extension} of~$W$ is a word of the form~$W_1xW_2$ where~$W_1$ and~$W_2$ are (possibly empty) words such that~$W=W_1W_2$ and~$x\in\A$ is a single letter. An \emph{extremal square-free word} is a square-free word such that none of its extensions are square-free.
	\end{definition}
	
	The only binary extremal square-free words are $010$ and $101$. Grytczuk, Kordulewski and Niewiadomski~\cite{GrytczukKordulewskiNiewiadomski} showed that there are infinitely many ternary extremal square-free words. Mol and Rampersad~\cite{MolRampersad} concretely determined all possible lengths of ternary extremal square-free words. The following conjecture was posed in~\cite{GrytczukKordulewskiNiewiadomski,MolRampersad}.
	
	\begin{conjecture}[{\cite{GrytczukKordulewskiNiewiadomski,MolRampersad}}] \label{conj:no-extremal-squarefree-4}
		There is no extremal square-free word over an alphabet of size at least~$4$.
	\end{conjecture}
	
	The notion of \emph{nonchalant words} was introduced and discussed in~\cite{GrytczukKordulewskiNiewiadomski,GrytczukKordulewskiPawlik} as a concept closely related to extremal square-free words. The sequence~$(N_0,N_1,\dots)$ of nonchalant words over a linearly ordered alphabet~$(\A,\prec)$ is generated as follows. Let~$N_0$ be the empty word, and for each integer~$i\in\N$ we define~$N_i$ to be the square-free extension~$PxS$ of~$N_{i-1}=PS$ such that the suffix~$S$ of~$N_{i-1}$ is shortest possible and the letter~$x\in\A$ is earliest possible with respect to~$\prec$. If no such $S$ and $x$ exist, then the algorithm terminates, i.e.\@ the sequence is finite. Note that this can happen only if $N_{i-1}$ is extremal.
	
	Even though it is known that there are infinitely many extremal words over alphabets of size 3, it is conjectured that the sequence of nonchalant words is infinite.
	
	\begin{conjecture}[\cite{GrytczukKordulewskiNiewiadomski,GrytczukKordulewskiPawlik}]\label{conj:nonchalant-infinite}
		For any alphabet~$\A$ of size~$k\ge3$ with a linear ordering~$\prec$, the sequence of nonchalant words over~$(\A,\prec)$ is infinite.
	\end{conjecture}
	
	Even if the sequence of nonchalant words is infinite, it still might not converge to an infinite word. Therefore, the following conjecture is even stronger than the previous one.
	
	\begin{conjecture}[\cite{GrytczukKordulewskiNiewiadomski,GrytczukKordulewskiPawlik}] \label{conj:nonchalant-convergence-3}
		For any alphabet~$\A$ of size~$k\ge3$ with a linear ordering~$\prec$, the sequence of nonchalant words over~$(\A,\prec)$ converges to an infinite word.
	\end{conjecture}
	
	Hong and Zhang~\cite{HongZhang} confirmed Conjectures~\ref{conj:no-extremal-squarefree-4} and~\ref{conj:nonchalant-convergence-3} for alphabets of size at least~$17$.
	
	\begin{theorem}[\cite{HongZhang}]
		For any integer~$k\ge17$, there is no extremal square-free word over an alphabet of size~$k$.
	\end{theorem}
	
	\begin{theorem}[\cite{HongZhang}]
		For any alphabet~$\A$ of size~$k\ge17$ with a linear ordering~$\prec$, the sequence of nonchalant words over~$(\A,\prec)$ converges to an infinite word.
	\end{theorem}
	
	Our main results improve on the results of Hong and Zhang and reduce the lower bound on the alphabet size from~$17$ to~$5$; this confirms Conjectures~\ref{conj:no-extremal-squarefree-4} and~\ref{conj:nonchalant-convergence-3} for alphabets of size at least~$5$.
	
	\begin{theorem} \label{thm:no-extremal-squarefree-5}
		For any integer~$k\ge5$, there is no extremal square-free word over an alphabet of size~$k$.
	\end{theorem}
	
	\begin{theorem} \label{thm:nonchalant-converges-5}
		For any alphabet~$\A$ of size~$k\ge5$ with a linear ordering~$\prec$, the sequence of nonchalant words over~$(\A,\prec)$ converges to an infinite word.
	\end{theorem}
	
	\section{The proof}
	
	We first introduce the notation and terminology we need, which are mostly identical to those in~\cite{HongZhang}. Write~$\N$ for the set of positive integers and~$\N_0$ for the set~$\N\cup\{0\}$. For~$m\in\N_0$, write~$[m]$ for the set~$\{1,\dots,m\}$ and~$[m]_0$ for the set~$[m]\cup\{0\}$. Let~$W$ be a word of length~$n$. Enumerate its letters from left to right as~$w_1,\dots,w_n$. For~$i\in[n]$, let~$W[i]$ denote the letter~$w_i$. Given integers~$1\le a<b\le n+1$, let~$W[a,b)$ denote the factor~$w_a\dots w_{b-1}$ of~$W$. For~$i\in[n-1]$, we call the space between letters~$w_i$ and~$w_{i+1}$ \emph{gap~$i$}; we call the spaces to the left of~$w_1$ and to the right of~$w_n$ \emph{gap~$0$} and \emph{gap~$n$} respectively.
	
	\begin{definition}[\cite{HongZhang}]
		Let~$W$ be a word of length~$n$ over an alphabet~$\A$. For~$b\in[n]_0$ and~$c\in\A$, let~$W+_bc$ be the word obtained by inserting the letter~$c$ at gap~$b$ in~$W$. For~$a\in[n]$, $\ell\in\N$, $b\in[n]_0$ and~$c\in\A$, the quadruple~$(a,\ell,b,c)$ is \emph{square-completing in~$W$} if~$(W+_bc)[a,a+\ell)=(W+_bc)[a+\ell,a+2\ell)$. The \emph{sign} of a square-completing quadruple~$(a,\ell,b,c)$ in~$W$ is~$1$ if~$b\le a+\ell-2$ and~$-1$ otherwise; this indicates whether the letter~$c$ we inserted at gap~$b$ appears in~$(W+_bc)[a,a+\ell)$ or~$(W+_bc)[a+\ell,a+2\ell)$.
	\end{definition}
	
	We introduce the key new concept of a \emph{minimal} square-completing quadruple. This helps us achieve better bounds in the proof and attain the value 5 in \Cref{thm:no-extremal-squarefree-5}.
	
	\begin{definition}
		A square-completing quadruple~$\q=(a,\ell,b,c)$ in a word~$W$ is \emph{minimal} if for every square-completing quadruple~$\q'=(a',\ell',b,c)$ in~$W$ one of the following statements holds.
		\begin{enumerate}[label=(M\arabic{*})]
			\item \label{item:mscq-min-ell} $\ell'>\ell$.
			\item \label{item:mscq-same-ell-split-signs} $\ell'=\ell$ and~$\q$ and~$\q'$ have different signs in~$W$.
			\item \label{item:mscq-same-ell-sign-positive} $\ell'=\ell$, both~$\q$ and~$\q'$ have sign~$1$ in~$W$ and~$a'\le a$.
			\item \label{item:mscq-same-ell-sign-negative} $\ell'=\ell$, both~$\q$ and~$\q'$ have sign~$-1$ in~$W$ and~$a'\ge a$.
		\end{enumerate}
	\end{definition}
	\begin{remark}
		The definition quickly implies the following:
		if $W$ is a word of length $n$ over an alphabet $\A$ and $(b,c)\in [n]_0\times \A$ is fixed, then if there is a square-completing quadruple of the form $(a,\ell,b,c)$, then there is also a minimal square-completing quadruple with $b$ and $c$ as the third and fourth entries. Furthermore, there are at most two minimal square-completing quadruples with $b$ and $c$ as the third and fourth entries, and if there are two, then both have the same second entry $\ell$ and different signs.
	\end{remark}
	
	Let~$W$ be a word. Write~$\cA(W)$ for the set of all minimal square-completing quadruples in~$W$. For~$\ell\in\N$, let~$\cA_\ell(W)$ be the set of all elements of~$\cA(W)$ with second entry~$\ell$; write~$\cA^{+}_\ell(W)$ for the set of all elements of~$\cA_\ell(W)$ with sign~$1$ and~$\cA^{-}_\ell(W)$ for the set of all elements of~$\cA_\ell(W)$ with sign~$-1$. 
	
	Our proof proceeds by assuming that an extremal square-free word $W$ over an alphabet of size at least 5 exists, deriving bounds on the size of $\cA(W)$, and showing that these bounds yield a contradiction. In our proof, it is convenient to sort the elements of~$\cA(W)$ according to the value of the second entry~$\ell$. We now state our main lemmas, which give bounds on various sums of sizes of sets~$\cA_{\ell}(W)$. The first lemma deals with sets~$\cA_{\ell}(W)$ with~$\ell\ge40$ and the second lemma handles~$2\le\ell\le39$.
	
	\begin{lemma}\label{lem:bound-40-infty}
		For every square-free word~$W$ of length~$n$, we have
		\[\sum_{\ell=40}^\infty\abs{\cA_\ell(W)}\le \frac{3n}{5}\,.\]
	\end{lemma}
	
	\begin{lemma}\label{lem:bounds-2-13-26-39}
		For every square-free word $W$ of length $n$, we have
		\[\sum_{\ell=2}^{12}\abs{\cA_{\ell}(W)}\le \frac{20n}{11}+\frac{140}{11}\,, \qquad \sum_{\ell=13}^{25}\abs{\cA_{\ell}(W)}\le \frac{3n}{10}\,,\qquad \textrm{ and} \qquad \sum_{\ell=26}^{39}\abs{\cA_{\ell}(W)} \le \frac{8n}{45}\,.\]
	\end{lemma}
	
	Now we shall apply \Cref{lem:bound-40-infty,lem:bounds-2-13-26-39} to prove \Cref{thm:no-extremal-squarefree-5,thm:nonchalant-converges-5}. We will prove \Cref{lem:bound-40-infty,lem:bounds-2-13-26-39} in \Cref{ssec:proof-bound-40-infty,ssec:proof-bounds-2-13-26-39}, respectively.
	
	\begin{proof}[Proof of \Cref{thm:no-extremal-squarefree-5}]
		Fix~$k\ge5$ and suppose that there is an extremal square-free word~$W$ of length~$n$ over an alphabet~$\A$ of size~$k$. Then for each~$(b,c)\in[n]_0\times\A$, there is a minimal square-completing quadruple~$(a,\ell,b,c)$ in~$W$; this gives~$\abs{\cA(W)}\ge k(n+1)$. On the other hand, grouping up the elements of~$\cA(W)$ according to their second entries gives~$\abs{\cA(W)}=\sum_{\ell\in\N}\abs{\cA_{\ell}(W)}$. Now by combining the bound in \Cref{lem:bound-40-infty} with the bounds in \Cref{lem:bounds-2-13-26-39} and observing that~$\abs{\cA_1(W)}=2n$, we get
		\begin{equation} \label{eq:mscq-collection-bounds}
			k\cdot(n+1)\le\abs{\cA(W)}=\sum_{\ell=1}^{\infty}\abs{\cA_{\ell}(W)} \le 2n + \frac{20n}{11}+\frac{140}{11} + \frac{3n}{10} + \frac{8n}{45} + \frac{3n}{5} < 4.896n + 12.73\,.
		\end{equation}
		
		We shall show that~$n\ge k+2$. First, note that any letter of~$\A$ which does not appear in~$W$ may be inserted at any gap of~$W$ without creating a square, so every letter of~$\A$ must appear in~$W$ and we have~$n\ge k$. Next, the absence of a repeated letter in~$W$ allows us to insert~$W[2]$ at gap~$0$ without creating a square, so~$W$ must contain a letter which appears at least twice and we have~$n\ge k+1$. Now suppose that~$n=k+1$. Then there is a unique letter~$a\in\A$ which appears more than once in~$W$; in fact, it appears exactly twice, say as~$W[i]$ and~$W[j]$ with~$i<j$. If~$i\ge2$, then we may insert~$a$ at gap~$0$ without creating a square; analogously, if~$j\le k$, then we may insert~$a$ at gap~$k+1$ without creating a square. Hence, we have~$i=1$ and~$j=k+1$. But now we may insert~$a$ at gap~$2\le k-1$ without creating a square, so we obtain a contradiction and conclude that~$n\ge k+2$.
		
		If~$k\ge6$, then by~\Cref{eq:mscq-collection-bounds} we have~$6(n+1)\le 4.896n+12.73$, so $n\le 6\le k$, which is a contradiction. Hence, we have~$k=5$. Now \Cref{eq:mscq-collection-bounds} implies~$n\le74$, so we have~$\cA_\ell(W)=\emptyset$ for all~$\ell>37$. By combining the bounds in \Cref{lem:bounds-2-13-26-39}, we get
		\[5\cdot(n+1)\le\sum_{\ell=1}^{37}\abs{\cA_\ell(W)} \le 2n + \frac{20n}{11} + \frac{140}{11} + \frac{3n}{10} + \frac{8n}{45} < 4.296n + 12.73\,,\]
		which implies $n\le 10$, so we have~$\cA_\ell(W)=\emptyset$ for all~$\ell>5$. Now we apply the first bound in \Cref{lem:bounds-2-13-26-39} to obtain
		\[5\cdot(n+1)\le\sum_{\ell=1}^{5}\abs{\cA_\ell(W)} \le 2n + \frac{20n}{11} + \frac{140}{11} < 3.819n + 12.73\,,\]
		which implies $n\le 6 = k+1$. This yields a contradiction and completes the proof.
	\end{proof}
	
	\begin{proof}[Proof of \Cref{thm:nonchalant-converges-5}]
		Fix a total ordering~$\prec$ of an alphabet~$\A$ of size~$k\ge5$ and let the sequence of nonchalant words over~$(\A,\prec)$ be~$(W_n)_{n\ge0}$. Let~$m\in\N$ and take~$n\ge49m+171$. By \Cref{thm:no-extremal-squarefree-5}, there is no extremal square-free word over $\A$. Hence, the sequence of nonchalant words is infinite and $W_n$ exists. Since the word~$W_n$ is square-free, by combining the bound in \Cref{lem:bound-40-infty} with the bounds in \Cref{lem:bounds-2-13-26-39} and observing that~$\abs{\cA_1(W_n)}=2n$, we get
		\[\abs{\cA(W_n)}=\sum_{\ell=1}^{\infty}\abs{\cA_{\ell}(W_n)} \le 2n + \frac{20n}{11}+\frac{140}{11} + \frac{3n}{10} + \frac{8n}{45} + \frac{3n}{5}< 4.896n + 12.73\,.\]
		Hence, $W_{n+1}$ is obtained from~$W_n$ by inserting a letter into one of the last~$\frac{4.896n + 17.73}{5}$ gaps of~$W_n$. In particular, this implies that~$W_{n+1}$ and~$W_n$ have the same prefix of length~$m\le n-\frac{4.896n + 17.73}{5}$. It follows that the prefix of~$(W_n)_{n\ge0}$ stabilises and so the sequence converges to an infinite word.
	\end{proof}
	
	\subsection{Proof of Lemma~\ref{lem:bound-40-infty}} \label{ssec:proof-bound-40-infty}
	
	In this subsection we prove \Cref{lem:bound-40-infty}. Our proof approach, which is motivated by the ideas of~\cite{HongZhang}, involves grouping up and analysing the interactions between pairs of minimal square-completing quadruples of the same sign. This motivates the following definition.
	
	\begin{definition}
		For positive integers $\ell,\ell'\in\N$ let $\f(\ell,\ell')\in \N_0\cup \{\infty\}$ be the least~$d$ such that there is an alphabet $\A$ and a square-free word $W$ over $\A$ with distinct minimal square-completing quadruples of sign 1 of the forms $(a,\ell,b,c)$ and $(a+d,\ell',b',c')$.
	\end{definition}
	
	The following lemma gives a simple upper bound on~$\f(\ell,\ell')$.
	
	\begin{lemma} \label{lem:gap-upper-bound}
		For all~$\ell,\ell'\in\N$ we have~$\f(\ell,\ell')\le\ell$.
	\end{lemma}
	
	\begin{proof}
		We distinguish three cases. 
		
		\textbf{Case 1:} $\ell\le\ell'$. Then define~$\A\coleq[\ell'+1]$, $P \coleq 2\cdots\ell$, $Q \coleq (\ell+1)\cdots\ell'$,~and consider the word~$W\coleq P1PQ(\ell'+1)PQ$. First, we show that it is square-free. Indeed, as the letters 1 and $\ell'+1$ only appear once in $W$, any square of $W$ must be a square in $P$ or~$PQ$. But in both of them, no letter appears more than once, so $W$ is square-free. Furthermore, $W$ has distinct minimal square-completing quadruples~$(1,\ell,0,1)$ and~$(\ell+1,\ell',\ell,\ell'+1)$ of sign~$1$, so~$\f(\ell,\ell')\le\ell$. 
		
		\textbf{Case 2:} $\ell'<\ell\le2\ell'$. Then set~$\A\coleq [\ell'+1]$, $R \coleq 2\cdots\ell'$, $S \coleq 2\cdots(\ell-\ell')$,~and consider the word~$W\coleq R(\ell'+1)S1R(\ell'+1)R$. Again, we claim that $W$ is square-free. Indeed, as the letter $1$ appears only once in $W$, any square of $W$ must be a square in $R(\ell'+1)S$ or $R(\ell'+1)R$. But in both of them, $\ell'+1$ only appears once because $\ell-\ell'< \ell'+1$, so each square of $W$ must be a square in $R$ or $S$. As both $R$ and $S$ have no square, $W$ is square-free as well. Furthermore, $W$ has distinct minimal square-completing quadruples~$(1,\ell,0,1)$ and~$(\ell+1,\ell',\ell,\ell'+1)$ of sign~$1$, so~$\f(\ell,\ell')\le\ell$. 
		
		\textbf{Case 3:} $2\ell'<\ell$. Set~$\A\coleq[\ell-\ell'+1]$, $T \coleq 2\cdots\ell'$, $U \coleq (\ell'+2)\cdots(\ell-\ell'+1)$, and consider the word~$W\coleq T(\ell'+1)TU1T(\ell'+1)TU$. Again, we claim that $W$ is square-free. Indeed, first the letter 1 is unique and then the letter $(\ell'+1)$ is unique in both of the resulting factors. Thus, any square of $W$ must be a square in $T$ or $TU$ which is impossible. Thus, $W$ is square-free. Furthermore, it has distinct minimal square-completing quadruples~$(1,\ell,0,1)$ and~$(\ell+1,\ell',\ell,\ell'+1)$ of sign~$1$, so~$\f(\ell,\ell')\le\ell$.
	\end{proof}
	
	The following lemma gives lower bounds on~$\f(\ell,\ell')$ and forms a key part of the proof of \Cref{lem:bound-40-infty}. Its proof, which entails a highly technical and careful analysis, is deferred to \Cref{ssec:proof-gap-lower-bound}.
	
	\begin{lemma} \label{lem:gap-lower-bound}
		For all~$\ell,\ell'\in\N\sm\{1\}$ we have
		\begin{equation*}
			\f(\ell,\ell') \ge
			\begin{cases}
				\ell &\textrm{if } \ell'=\ell \\
				4\ell-3\ell'+2 &\textrm{if } \ell+1<\ell'\le\frac{4}{3}\ell \\
				1 &\textrm{if } \frac{4}{3}\ell<\ell'<\frac{3}{2}\ell \\
				\ell' &\textrm{if } \frac{2}{3}\ell<\ell'<\ell-1 \\
				0 &\textrm{otherwise.}
			\end{cases}
		\end{equation*}
	\end{lemma}
	
	We note that even though we only prove lower bounds on $\f(\ell,\ell')$, computational evidence suggests that the lower bounds in \Cref{lem:gap-lower-bound} are actually the precise values of $\f(\ell,\ell')$.
	
	Now we shall develop the machinery needed to use \Cref{lem:gap-lower-bound} to prove \Cref{lem:bound-40-infty}. We start with the following simple observation.
	
	\begin{observation} \label{obs:no-late-inserts}
		If~$(a,\ell,b,c)$ is a square-completing quadruple in a square-free word~$W$ of length~$n$, then~$a\le n+2-2\ell$.
	\end{observation}
	
	\begin{proof}
		Since~$(a,\ell,b,c)$ is a square-completing quadruple, we know that~$(W+_{b}c)[a,a+\ell)=(W+_{b}c)[a+\ell,a+2\ell)$. In particular, $(W+_{b}c)[a+\ell,a+2\ell)$ is a proper factor of the word~$W+_{b}c$ of length~$n+1$. Thus, $a+2\ell-1\le n+1$, so~$a\le n+2-2\ell$.
	\end{proof}
	
	For a finite set~$S\subs\N\sm\{1\}$, let~$G(S)$ be the directed graph with vertex set~$S$ and edge set~$S\times S$, i.e.\@ all directed edges, including loops. The following lemma gives an upper bound on~$\sum_{\ell\in S}\abs{\cA^{+}_\ell(W)}$ in terms of the minimum mean weight of a directed cycle in the weighted digraph~$(G(S),\w)$. Here, the \emph{mean weight} of a directed cycle with edge sequence $(e_1,\dots,e_k)$, where $w_i$ is the weight of edge $e_i$ for $i\in[k]$, is $\frac{1}{k}\sum_{i\in [k]}w_i$.
	
	\begin{lemma} \label{lem:bound-min-mean-weight-cycle}
		Let~$S\subs\N\sm\{1\}$ be a finite set. Let~$\bw$ denote the minimum mean weight of a directed cycle in the weighted digraph $(G(S),\w)$. Suppose that~$\bw\ne0$. Then for every square-free word~$W$ of length~$n$ we have
		\[ \sum_{\ell\in S}\abs{\cA^{+}_\ell(W)}\le \max\!\left(0,\frac{n+\max S-2\min S+1}{\bw}\right)\,. \]
	\end{lemma}
	
	\begin{proof}
		Enumerate the elements of~$\bigcup_{\ell\in S}\cA^{+}_\ell(W)$ as~$(a_1,\ell_1,b_1,c_1),\dots,(a_m,\ell_m,b_m,c_m)$ such that~$a_1\le \dots\le a_m$. The desired statement holds trivially when~$m=0$, so we may assume that~$m\in\N$. Let $T$ be the closed directed walk in $G(S)$ with vertex sequence $(\ell_1,\dots,\ell_m,\ell_1)$. For a closed directed walk~$C$ in~$G(S)$ with vertex sequence $(u_1,\dots,u_c,u_1)$, set~$\w(C) \coleq \w(u_c,u_1) + \sum_{i=1}^{c-1} \w(u_i,u_{i+1})$. We prove the following claim.
		
		\begin{claim} \label{claim:start-span-word}
			$m\cdot\bw \le \w(T) \le n+\max S-2\min S+1$.
		\end{claim}
		
		\begin{claimproof}
			Fix a decomposition of the multiset~$E(T)=\bigcup_{i=1}^sE(C_i)$ where $C_1,\dots,C_s$ are directed cycles. Since~$\bw$ is the minimum mean weight of a directed cycle, for all~$i\in[s]$ we have~$\w(C_i)\ge e(C_i)\cdot\bw$. Hence, we obtain~$\w(T) = \sum_{i=1}^s\w(C_i) \ge \sum_{i=1}^se(C_i)\cdot\bw = m\cdot\bw$. For the upper bound, we start with the case~$m=1$. Here we have~$T=(\ell_1)$ and~$n\ge2\ell_1-1$, which gives
			\[ \w(T)=\f(\ell_1,\ell_1)\overset{\ref{lem:gap-upper-bound}}\le\ell_1\le n-\ell_1+1\le n+\max S-2\min S+1\,. \]
			It remains to consider the case~$m\ge2$. Since~$W$ is square-free, by the definition of~$\f$ for all~$i\in[m-1]$ we have~$\w(\ell_{i},\ell_{i+1}) = \f(\ell_{i},\ell_{i+1}) \le a_{i+1}-a_i$. Hence, we have
			\begin{align*}
				\w(T) &= \w(\ell_m,\ell_1) + \sum_{i=1}^{m-1}\w(\ell_i,\ell_{i+1}) \le \w(\ell_m,\ell_1) + a_m-a_1 \le\w(\ell_m,\ell_1) + a_m-1\\
				&\overset{\ref{lem:gap-upper-bound}}\le  \max S+a_m-1\overset{\ref{obs:no-late-inserts}}\le \max S+n+2-2\min S-1= n+\max S-2\min S+1\,,
			\end{align*}
			which concludes the proof of the claim.
		\end{claimproof}
		
		To complete the proof, note that \Cref{claim:start-span-word} gives~$m\le\frac{n+\max S-2\min S+1}{\bw}$.
	\end{proof}
	
	The following lemma gives an upper bound on~$\sum_{\ell\in S}\abs{\cA^{+}_\ell(W)}$ for short intervals~$S$ of positive integers of the same parity. This serves as a single-sign analogue of \Cref{lem:bound-40-infty} and its proof involves the application of \Cref{lem:gap-lower-bound,lem:bound-min-mean-weight-cycle}.
	
	\begin{lemma}\label{lem:bound-same-parity-twothirds}
		Let~$W$ be a square-free word of length~$n$. Suppose that~$a,b\in\N\setminus\{1\}$ have the same parity and satisfy~$a\le b<\frac{3}{2}a$. Then
		\[\sum_{i=0}^{\frac{b-a}{2}}\abs{\cA^{+}_{a+2i}(W)}\le \max\!\left(0,\frac{2(n+b+1-2a)}{a+1}\right)\,.\]
	\end{lemma}
	
	\begin{proof}
		Let $S\coleq\{a,a+2,a+4,\dots,b\}$. By \Cref{lem:bound-min-mean-weight-cycle}, it is enough to show that the minimum mean weight of $(G(S),\w)$ is at least $\frac{a+1}{2}$. To do this, we define the function~$\pi\colon V(G(S))\to\R$ as follows.
		\begin{equation*}
			\pi(u)\coleq
			\begin{cases}
				\frac{a+1}{2}-1 &\textrm{if }u<\frac{3}{4}b \\
				\frac{a+1}{2}-(4u-3b+2) &\textrm{if }\frac{3}{4}b\le u< \frac{a-3+6b}{8} \\
				0 &\textrm{if }\frac{a-3+6b}{8}\le u\le b\,.
			\end{cases}
		\end{equation*}
		We prove the following key claim whose proof entails detailed but elementary calculations with extensive case distinction.
		
		\begin{claim} \label{claim:mean-weight-bound-vtx-pot}
			For all~$(\ell,\ell')\in E(G(S))$ we have~$\w^{\pi}(\ell,\ell')\coleq\w(\ell,\ell')+\pi(\ell)-\pi(\ell')\ge\frac{a+1}{2}$.
		\end{claim}
		Before proving the claim, we will quickly see how it implies that the minimum mean weight of $(G(S),\w)$ is at least $\frac{a+1}{2}$. Let~$C=v_1\dots v_{\ell}v_1$ be a directed cycle in~$G(S)$; set~$v_{\ell+1}\coleq v_1$ for convenience. Then, 
		\[ \frac{1}{\ell}\sum_{i=1}^{\ell}\w(v_i,v_{i+1}) = \frac{1}{\ell}\sum_{i=1}^{\ell}\left(\w(v_i,v_{i+1})+\pi(v_i)-\pi(v_{i+1})\right) \overset{\ref{claim:mean-weight-bound-vtx-pot}}\ge \frac{a+1}{2}\,, \]
		so the minimum mean weight of $(G(S),\w)$ is at least $\frac{a+1}{2}$. Thus, we only need to prove the claim.
		\begin{claimproof}
			Note that since~$u<\frac{a-3+6b}{8}$ is equivalent to $0<\frac{a+1}{2}-(4u-3b+2)$ and~$\frac{3}{4}b\le u$ is equivalent to $\frac{a+1}{2}-(4u-3b+2)\le \frac{a+1}{2}-2$, for all~$u\in S$ we have~$0\le\pi(u)\le \frac{a+1}{2}-1$ and for all~$u,u'\in S$ with~$u\le u'$ we have~$\pi(u)\ge \pi(u')$.
			
			Let~$\ell,\ell'\in S$. If~$\ell=\ell'$, then by \Cref{lem:gap-lower-bound} we have~$\w^{\pi}(\ell,\ell')=\w(\ell,\ell')\ge \ell \ge\frac{a+1}{2}$. If~$\ell>\ell'$, then since~$\ell$ and~$\ell'$ have the same parity by the definition of~$S$, we have~$\ell'<\ell-1$. Furthermore, we have~$\frac{2}{3}\ell\le\frac{2}{3}b<a\le \ell'$. Hence, by \Cref{lem:gap-lower-bound} we have
			\[ \w^\pi(\ell,\ell') = \w(\ell,\ell')+\pi(\ell)-\pi(\ell') \ge \ell'+0-\left(\frac{a+1}{2}-1\right) \ge \frac{a+1}{2}\,.\]
			From now onwards, we have~$\ell<\ell'$. Since~$\ell$ and~$\ell'$ have the same parity, we have~$\ell<\ell'-1$. Furthermore, we have~$\frac{2}{3}\ell'\le\frac{2}{3}b<a\le\ell$. We distinguish two main cases.
			
			\textbf{Case 1:}~$\ell<\frac{3}{4}\ell'$. Since~$\ell<\frac{3}{4}\ell'\le\frac{3}{4}b$, we have~$\pi(\ell)=\frac{a+1}{2}-1$. Since~$b<\frac{3}{2}a$ and~$a\le\ell<\frac{3}{4}\ell'$, we have~$\frac{a-3+6b}{8} <\frac{a+6\cdot\frac{3}{2}a}{8}<\frac{4}{3}a<\ell'$. Hence, we have~$\pi(\ell')=0$. Now by \Cref{lem:gap-lower-bound} we have~$\w(\ell,\ell')\ge1$, so we obtain
			\[ \w^\pi(\ell,\ell') = \w(\ell,\ell') + \pi(\ell) - \pi(\ell') \ge 1 + \left(\frac{a+1}{2} - 1\right) - 0 =\frac{a+1}{2}\,. \]
			
			\textbf{Case 2:}~$\ell\ge\frac{3}{4}\ell'$. In this case, by \Cref{lem:gap-lower-bound} we have~$\w(\ell,\ell')\ge 4\ell-3\ell'+2\ge 2$. We distinguish three subcases.
			
			\textbf{Subcase 2A:}~$\ell<\frac{3}{4}b$. In this subcase, we have~$\pi(\ell)=\frac{a+1}{2}-1$. If~$\frac{a-3+6b}{8}\le\ell'$, then we have~$\pi(\ell')=0$ and so
			\[ \w^\pi(\ell,\ell')=\w(\ell,\ell')+\pi(\ell)-\pi(\ell')\ge 2+\left(\frac{a+1}{2}-1\right)-0\ge\frac{a+1}{2}\,. \]
			If~$\frac{3}{4}b\le\ell'<\frac{a-3+6b}{8}$, then we have~$\pi(\ell')=\frac{a+1}{2}-(4\ell'-3b+2)$ and so
			\begin{align*}
				\w^\pi(\ell,\ell') &= \w(\ell,\ell')+\pi(\ell)-\pi(\ell') \\
				&\ge (4\ell-3\ell'+2)+\left(\frac{a+1}{2}-1\right)-\left(\frac{a+1}{2}-(4\ell'-3b+2)\right)\\
				&=4\ell+\ell'-3b+3\ge5a-\frac{9}{2}a+3\ge\frac{a+1}{2}\,,
			\end{align*}
			where we used~$a\le\ell\le\ell'$ and~$b<\frac{3}{2}a$. Otherwise, we have~$\ell'<\frac{3}{4}b$, in which case we have~$\pi(\ell')=\frac{a+1}{2}-1$ and so
			\[ \w^\pi(\ell,\ell') = \w(\ell,\ell')+\pi(\ell)-\pi(\ell') \ge 4\ell-3\ell'+2 \ge 4a-\frac{27}{8}a +3 \ge\frac{a+1}{2}\,,\]
			where we use~$a\le\ell$, $\ell'<\frac{3}{4}b$ and~$b<\frac{3}{2}a$ to obtain the penultimate inequality.
			
			\textbf{Subcase 2B:}~$\frac{3}{4}b\le\ell<\frac{a-3+6b}{8}$. In this subcase, we have~$\pi(\ell)=\frac{a+1}{2}-(4\ell-3b+2)$. If~$\frac{a-3+6b}{8}\le\ell'$, then we have~$\pi(\ell')=0$ and so
			\begin{align*}
				\w^\pi(\ell,\ell') &= \w(\ell,\ell')+\pi(\ell)-\pi(\ell') \\
				&\ge (4\ell-3\ell'+2)+\left(\frac{a+1}{2}-(4\ell-3b+2)\right)-0 \ge\frac{a+1}{2}\,,
			\end{align*}
			where we simplify expressions and use~$\ell'\le b$ in the final inequality. Otherwise, we have~$\frac{3}{4}b\le\ell'<\frac{a-3+6b}{8}$, in which case we have~$\pi(\ell')=\frac{a+1}{2}-(4\ell'-3b+2)$ and so
			\begin{align*} 
				\w^\pi(\ell,\ell')&= \w(\ell,\ell')+\pi(\ell)-\pi(\ell') \ge 4\ell-3\ell'+2-4\ell+4\ell'=\ell'+2 \ge \frac{a+1}{2}\,, 
			\end{align*}
			where we use~$a\le\ell'$ in the final inequality.
			
			\textbf{Subcase 2C:}~$\frac{a-3+6b}{8}\le\ell$. In this case, we have~$\pi(\ell)=\pi(\ell')=0$. Hence, we have
			\[ \w^\pi(\ell,\ell')= \w(\ell,\ell')+\pi(\ell)-\pi(\ell') \ge 4\ell-3\ell'+2 \ge \frac{a-3+6b}{2}-3\ell'+2 \ge \frac{a+1}{2}\,, \]
			where we used~$\frac{a-3+6b}{8}\le\ell\le\ell'\le b$.
		\end{claimproof}
	\end{proof}
	
	The proof of \Cref{lem:bound-40-infty} applies the following observation, which shows that bounds on~$\sum_{\ell\in S}\abs{\cA^{+}_{\ell}(W)}$ also apply to~$\sum_{\ell\in S}\abs{\cA^{-}_{\ell}(W)}$ and vice versa. 
	
	\begin{observation} \label{obs:flip-sign-bounds}
		Given~$m,n\in\N$ and~$S\subseteq\N$, the following statements are equivalent.
		\begin{enumerate}
			\item \label{item:positive-sign-bound} For every square-free word~$W$ of length~$n$, we have~$\sum_{\ell\in S}\abs{\cA^{+}_{\ell}(W)}\le m$.
			\item \label{item:negative-sign-bound} For every square-free word~$W$ of length~$n$, we have~$\sum_{\ell\in S}\abs{\cA^{-}_{\ell}(W)}\le m$.
		\end{enumerate}
	\end{observation}
	
	\begin{proof}
		Suppose that~\ref{item:positive-sign-bound} holds. Fix a square-free word~$W$ of length~$n$ and let~$W'$ be the square-free word of length~$n$ defined for all~$i\in[n]$ as~$W'[i]\coleq W[n-i+1]$. Then for each~$\ell\in S$ we shall show that~$\abs{\cA^{-}_{\ell}(W)}\le\abs{\cA^{+}_{\ell}(W')}$ by explicitly defining an injective map~$f_\ell\colon\cA^{-}_{\ell}(W)\to\cA^{+}_{\ell}(W')$. Take~$(a,\ell,b,c)\in\cA^{-}_{\ell}(W)$ and set~$a_0\coleq n+3-a-2\ell$. Since~$(a,\ell,b,c)\in\cA^{-}_{\ell}(W)$, the quadruple~$(a_0,\ell,n-b,c)$ is square-completing in~$W'$ with sign~$1$. Hence, there exists $(a',\ell',n-b,c)\in\cA^{+}_{\ell'}(W')$ such that~$\ell'<\ell$ or both~$\ell'=\ell$ and~$a_0\le a'$. Suppose that~$\ell'<\ell$. Then the quadruple~$(n+3-a'-2\ell',\ell',b,c)$ is square-completing in~$W$ with sign~$-1$, contradicting the minimality of~$(a,\ell,b,c)$. Hence, we have~$\ell'=\ell$ and~$a_0\le a'$. Set~$f_{\ell}(a,\ell,b,c)=(a',\ell,n-b,c)$; this is well-defined and yields the desired injective map. Hence, we have~$\abs{\cA^{-}_{\ell}(W)}\le\abs{\cA^{+}_{\ell}(W')}$. Now take a sum over~$\ell\in S$ to obtain~$\sum_{\ell\in S}\abs{\cA^{-}_{\ell}(W)} \le \sum_{\ell\in S}\abs{\cA^{+}_{\ell}(W')} \le m$, where the second inequality follows from~\ref{item:positive-sign-bound}. Hence, we conclude that~\ref{item:negative-sign-bound} holds.
		
		Finally, note that the proof that~\ref{item:negative-sign-bound} implies~\ref{item:positive-sign-bound} is entirely analogous.
	\end{proof}
	
	We now give a proof of \Cref{lem:bound-40-infty} using \Cref{lem:bound-same-parity-twothirds} and \Cref{obs:flip-sign-bounds}.
	
	\begin{proof}[Proof of \Cref{lem:bound-40-infty}]
		Take~$a\in\N$ and the largest positive integer~$b<\frac{3a}{2}$ which has the same parity as~$a$. By \Cref{lem:bound-same-parity-twothirds} for both~$(a,b)$ and~$(a+1,b+1)$, we obtain
		\[ \sum_{\ell\in \left[a,\frac{3}{2}a\right)}\abs{\cA^{+}_\ell(W)} \le \sum_{i=0}^{\frac{b-a}{2}}\abs{\cA^{+}_{a+2i}(W)} + \sum_{i=0}^{\frac{b-a}{2}}\abs{\cA^{+}_{a+2i+1}(W)} \le \frac{2n}{a+1} + \frac{2n}{a+2} \le \frac{4n}{a}\,. \]
		Now we take~$a=\ceil{40\cdot\left(\frac{3}{2}\right)^j}$ for all~$j\in\N_0$ and obtain
		\begin{equation*}
			\sum_{\ell=40}^\infty\abs{\cA^{+}_\ell(W)} = \sum_{j=0}^\infty\sum_{\ell\in\left[40\cdot \left(\frac{3}{2}\right)^j,40\cdot \left(\frac{3}{2}\right)^{j+1}\right)}\abs{\cA^{+}_\ell(W)}\le \sum_{j=0}^\infty\frac{4n}{40\cdot\left(\frac32\right)^j} =\frac{3n}{10}\,.
		\end{equation*}
		Finally, since~$\abs{\cA_\ell(W)}=\abs{\cA^{+}_\ell(W)}+\abs{\cA^{-}_\ell(W)}$ for all~$\ell\in\N$, the desired outcome follows by \Cref{obs:flip-sign-bounds}.
	\end{proof}
	
	\subsection{Proof of Lemma~\ref{lem:bounds-2-13-26-39}} \label{ssec:proof-bounds-2-13-26-39}
	
	In this subsection we prove \Cref{lem:bounds-2-13-26-39}. We remark that the combination of \Cref{lem:bound-same-parity-twothirds} and \Cref{obs:flip-sign-bounds} used to prove \Cref{lem:bound-40-infty} may be applied to prove upper bounds on sums similar to those in \Cref{lem:bounds-2-13-26-39}, but the bounds are substantially worse. Since the main contribution to the key sum in \Cref{eq:mscq-collection-bounds} comes from the terms with~$\ell\le39$ rather than~$\ell\ge40$, it makes sense to seek better bounds for them by analysing interactions between minimal square-completing quadruples which are more complex than the pairwise interactions studied in the previous subsection. We start with the following definition.
	
	\begin{definition}\label{def:realise}
		Let~$k\in\N$ and for~$i\in[k]$, let~$S_i\subs\N\sm\{1\}$. Let~$W$ be a square-free word of length~$n$ and let~$(a_1,\ell_1,b_1,c_1),\dots,(a_m,\ell_m,b_m,c_m)$ be an enumeration of the elements of~$\cA^{-}(W)\sm\cA^{-}_1(W)$. For~$a\in\Z$, set~$L_a = \{\ell_i:a_i=a\}$ if~$a\in[n]$ and~$L_a\coleq\emptyset$ otherwise. We say that~$W$ \emph{realises} the~$k$-tuple $(S_1,\dots,S_k)$ if there exists some~$a\in[-k+1,n-1]$ such that for all~$i\in[k]$, we have~$S_i\subs L_{i+a}$; we say that the~$k$-tuple~$(S_1,\dots,S_k)$ is \emph{realisable} if there is a square-free word~$W$ that realises~$(S_1,\dots,S_k)$.
	\end{definition}
	
	For a finite set~$S\subs\N\sm\{1\}$ and~$k\in\N$, define the directed graph~$G_k(S)$ together with an edge weight function~$\g\colon E\big(G_k(S)\big)\to\N_0$ as follows. The vertex set of~$G_k(S)$ is the set~$R_{S,k}$ of all realisable~$k$-tuples~$(S_1,\dots,S_k)$ such that for all~$i\in[k]$ we have~$S_i\subs S$. An ordered pair~$(T,T')\in R_{S,k}\times R_{S,k}$ with~$T=(S_1,\dots,S_k)$ and~$T'=(S_1',\dots,S_k')$ is an edge of~$G_k(S)$ if and only if~$S_{i+1}=S_i'$ for all~$i\in[k-1]$; if~$(T,T')\in E\big(G_k(S)\big)$, we set the weight~$\g\big((T,T')\big)$ to be~$\abs{S_k'}$.
	
	\begin{lemma} \label{lem:bound-max-mean-weight-cycle}
		Given a finite set~$S\subs\N\sm\{1\}$ and an integer~$k\in\N$, write~$\bwsk$ for the maximum mean weight of a directed cycle in the weighted digraph $(G_k(S),\g)$. Then for every square-free word~$W$ of length~$n$ we have
		\[\sum_{\ell\in S}\abs{\cA^{-}_{\ell}(W)}\le \max\big(0,(n+k-2\min S+1)\cdot\bwsk\big).\]
	\end{lemma}
	
	\begin{proof}
		Fix a finite set~$S\subs\N\sm\{1\}$, an integer~$k\in\N$, a square-free word~$W$ of length~$n$ and enumerate the elements of~$\cA^{-}(W)\sm\cA^{-}_1(W)$ as~$(a_1,\ell_1,b_1,c_1),\dots,(a_m,\ell_m,b_m,c_m)$. For~$a\in\Z$ set~$L_a = \{\ell_i:a_i=a\}$ if~$a\in[n]$ and~$L_a\coleq\emptyset$ otherwise; set~$S_a\coleq L_a\cap S$. By \Cref{obs:no-late-inserts}, we have~$S_a=\emptyset$ for all~$a>n'\coleq n+2-2\min S$. It follows that if $n'\le0$ then $\sum_{\ell\in S}\abs{\cA^-_\ell(W)}=0$, in which case the desired statement holds trivially. Hence, we may assume that $n'\in\N$. Set~$n^*:=n'+k-1$.
		
		For~$i\in[n^*]$ set~$S'_i,S'_{n^*+i}\coleq S_i$ and~$X_i\coleq(S'_i,S'_{i+1},\dots,S'_{i+k-1})$. Note that the vertex sequence~$(X_1,\dots,X_{n^*},X_1)$ defines a closed directed walk~$T$ in~$G_k(S)$. Indeed, since~$W$ realises the~$k$-tuples~$X_1,\dots,X_{n^*}$, they are all vertices of~$G_k(S)$. Furthermore, it follows from the definitions of the~$k$-tuples~$X_i$ and~$E(G_k(S))$ that~$(X_{n^*},X_1)\in E(G_k(S))$ and $(X_i,X_{i+1})\in E(G_k(S))$ for all $i\in[n^*-1]$. For a closed directed walk~$C$ in~$G_k(S)$ with vertex sequence $(U_1,\dots,U_c,U_1)$ set~$\g(C) \coleq \g(U_c,U_1) + \sum_{i=1}^{c-1} \g(U_i,U_{i+1})$. Note that~$\g(T)=\sum_{a\in[n]}\abs{S_a}=\sum_{\ell\in S}\abs{\cA^{-}_{\ell}(W)}$.
		
		Fix a decomposition of the multiset~$E(T)=\bigcup_{i=1}^sE(C_i)$ where $C_1,\dots,C_s$ are directed cycles. Since~$\bwsk$ is the maximum mean weight of a directed cycle, for all~$i\in[s]$ we have~$\g(C_i)\le e(C_i)\cdot\bwsk$. Note that~$\sum_{i=1}^se(C_i)=e(T)=n^*$. Hence, we obtain
		\begin{align*}
			\sum_{\ell\in S}\abs{\cA^{-}_{\ell}(W)} &= \g(T) = \sum_{i=1}^s\g(C_i) \le \left(\sum_{i=1}^se(C_i)\right)\cdot\bwsk = n^*\cdot\bwsk \\
			&= (n+k-2\min S+1)\cdot\bwsk\,,
		\end{align*}
		completing the proof.
	\end{proof}
	
	\begin{lemma} \label{lem:max-mean-weight-cycles}
		For
		\begin{align*}
			(k_1,S_1)&\coleq(10,\{2,\dots,12\}),\\
			(k_2,S_2)&\coleq(25,\{13,\dots,25\}),\\
			(k_3,S_3)&\coleq(34,\{26,\dots,39\}),
		\end{align*}
		we have
		\[
		\bwp{S_1}{k_1}=\frac{10}{11},\qquad
		\bwp{S_2}{k_2}=\frac{3}{20},\qquad
		\bwp{S_3}{k_3}=\frac{4}{45}.
		\]
	\end{lemma}
	
	\begin{proof}
		The statement was verified using a C++ program. For each $i\in[3]$, it first builds the weighted digraph $(G_{k_i}(S_i),\psi_{S_i,k_i})$ and then computes the maximum mean weight cycle using Howard's Algorithm. The source code can be found under \url{https://github.com/JEval42/NoExtremalSquareFreeWordsOverAlphabetsOfSizeAtLeast5/}. The calculations were performed on a CPU node of the Curta high-performance computing system at Freie Universität Berlin. Each instance was executed serially using one allocated CPU core. The program was compiled in C++17 mode with GCC 13.3.0 using the flags 
		\begin{verbatim}-O3 -march=native -flto -DNDEBUG.\end{verbatim} 
		The three computations together required 16 hours of wall-clock time and reached a peak resident memory usage of 23.15 GB.
	\end{proof}
	
	Finally, we are ready to prove \Cref{lem:bounds-2-13-26-39}.
	
	\begin{proof}[Proof of \Cref{lem:bounds-2-13-26-39}]
		Let~$S_1\coleq\{2,\dots,12\}$, $S_2\coleq\{13,\dots,25\}$ and~$S_3\coleq\{26,\dots,39\}$. The combination of \Cref{lem:bound-max-mean-weight-cycle} applied with~$(k_1,S_1)$, $(k_2,S_2)$ and~$(k_3,S_3)$ together with the values of~$\bwp{S_1}{k_1}$, $\bwp{S_2}{k_2}$ and~$\bwp{S_3}{k_3}$ returned by \Cref{lem:max-mean-weight-cycles} yields
		\[\sum_{\ell=2}^{12}\abs{\cA^{-}_{\ell}(W)}\le \frac{10n}{11}+\frac{70}{11}\,, \qquad \sum_{\ell=13}^{25}\abs{\cA^{-}_{\ell}(W)}\le \frac{3n}{20}\,,\qquad \textrm{ and} \qquad \sum_{\ell=26}^{39}\abs{\cA^{-}_{\ell}(W)} \le \frac{4n}{45}\,.\]
		Finally, since~$\abs{\cA_\ell(W)}=\abs{\cA^{+}_\ell(W)}+\abs{\cA^{-}_\ell(W)}$ for all~$\ell\in\N$, the desired outcome follows by \Cref{obs:flip-sign-bounds}.
	\end{proof}
	
	\subsection{Proof of Lemma~\ref{lem:gap-lower-bound}} \label{ssec:proof-gap-lower-bound}
	
	In this subsection, we prove \Cref{lem:gap-lower-bound}. We start with a simple technical fact. It says that in a square-completing quadruple of sign 1 in a square-free word $W$, the inserted letter cannot be inserted as the last letter of the first half of the resulting square. 
	
	\begin{lemma} \label{lem:scq-insertion-no-right-edge}
		If~$(a,\ell,b,c)$ is a square-completing quadruple of sign~$1$ in a square-free word~$W$ with $\ell\ge 2$, then~$b\le a+\ell-3$.
	\end{lemma}
	
	\begin{proof}
		By the definition of sign 1, we already have $b\le a+\ell-2$. Suppose that~$b=a+\ell-2$. Then~$W[a,a+\ell-1)=(W+_bc)[a,a+\ell-1)=(W+_bc)[a+\ell,a+2\ell-1)=W[a+\ell-1,a+2\ell-2)$ gives a square in~$W$ since $\ell\ge 2$, which is a contradiction.
	\end{proof}
	
	The following two lemmas represent the technical core of the proof of \Cref{lem:gap-lower-bound}; the first lemma deals with~$\ell=\ell'$ and the second lemma handles~$\ell'\ne\ell$. Each statement in these lemmas is derived by showing that it being false would contradict the square-free nature of the given word or the minimality of the given square-completing quadruples.
	
	\begin{lemma} \label{lem:same-length-mscq-outcomes}
		The following hold for any pair $(a,\ell,b,c)$ and $(a',\ell,b',c')$ of distinct minimal square-completing quadruples of sign 1 in a square-free word $W$.
		\begin{enumerate}
			\item \label{item:same-length-mscq-insertion-offset} $b\ne b'$.
			\item \label{item:same-length-mscq-insertion-no-left-edge} If $b'>b$ and $a'\le a+\ell-1$, then $b'\ge a'$.
			\item \label{item:same-length-mscq-spacing} If $b'>b$, $b\le a+\ell-3$ and $b'\ge a'$, then $a'\ge a+\ell$.
		\end{enumerate}
	\end{lemma}
	
	\begin{proof}
		(i) Suppose that $b=b'$. We have $c=(W+_bc)[b+1]=(W+_bc)[b+\ell+1]=W[b+\ell]$ and $c'=(W+_bc')[b+1]=(W+_bc')[b+\ell+1]=W[b+\ell]$, so $c=c'$. Since $(a,\ell,b,c)$ and $(a',\ell,b,c)$ are distinct and minimal, we have $a<a'<a$, which is a contradiction.
		
		(ii) By the definition of square-completing quadruple, we already have $b'\ge a'-1$. Suppose that $b'>b$, $a'\le a+\ell-1$ and $b'=a'-1$. Since $b<b'=a'-1\le a+\ell-2$, we have $W[a'-1]=(W+_bc)[a']=(W+_bc)[a'+\ell]=W[a'+\ell-1]$. Furthermore, we have $W[a',a'+\ell-1)=(W+_{b'}c')[a'+1,a'+\ell)=(W+_{b'}c')[a'+\ell+1,a'+2\ell)=W[a'+\ell,a'+2\ell-1)$. But this implies that $W[a'-1,a'+\ell-1)=W[a'+\ell-1,a'+2\ell-1)$ yields a square in $W$, which is a contradiction.
		
		(iii) Suppose that $b'>b$, $b\le a+\ell-3$, $b'\ge a'$ and $a'\le a+\ell-1$. Set $i\coleq\max(a',b+1)$. Since $a'\le i\le b'$, we have $W[i]=(W+_{b'}c')[i]=(W+_{b'}c')[i+\ell]=W[i+\ell-1]$. Suppose further that $a'\le a+\ell-2$. Since $b<i\le a+\ell-2$, we have $W[i]=(W+_bc)[i+1]=(W+_bc)[i+\ell+1]=W[i+\ell]$. But now $W[i+\ell]=W[i+\ell-1]$ yields a square in $W$, which is a contradiction. Hence, we have $a'=a+\ell-1$. Now $b < i = a' = a+\ell-1$, so we have $W[i+\ell-1]=(W+_bc)[i+\ell]=(W+_bc)[i]=W[i-1]$. But $W[i]=W[i-1]$ gives a square in $W$, which is a contradiction.
	\end{proof}
	
	\begin{lemma} \label{lem:diff-length-mscq-outcomes}
		The following hold for any pair~$(a,\ell,b,c)$ and~$(a',\ell',b',c')$ of distinct minimal square-completing quadruples of sign 1 with~$\ell'-\ell-1>0$ in a square-free word~$W$.
		\begin{enumerate}
			\item \label{item:diff-length-mscq-left} $\max(a,a')-1 + \ell'-\ell > \min(b,b')$.
			\item \label{item:diff-length-mscq-right} $\max(b,b') + \ell'-\ell > \min(a+\ell,a'+\ell')-2$.
			\item \label{item:diff-length-mscq-mid} $\max(a'-1,b) + \ell'-\ell-1 > \min(a+\ell-2,b')$.
			\item \label{item:diff-length-mscq-mid-flip} $\max(a-1,b') + \ell'-\ell+1 > \min(a'+\ell'-2,b)$.
			\item \label{item:diff-length-mscq-shifted-left} $\max(a-\ell'+\ell,a')-1 + \ell'-\ell > \min(b-\ell'+\ell,b')$.
			\item \label{item:diff-length-mscq-shifted-right} $\max(b-\ell'+\ell,b') + \ell'-\ell > \min(a+2\ell-\ell',a'+\ell')-2$.
			\item \label{item:diff-length-mscq-shifted-mid} $\max(a'-1,b-\ell'+\ell+1) + \ell'-\ell-1 > \min(a+2\ell-\ell'-1,b')$.
			\item \label{item:diff-length-mscq-augmented-right} $\max(b,b',a'+\ell'-\ell-2)+1 + \ell'-\ell > \min(a+\ell,a'+2\ell'-\ell)-1$.
			\item \label{item:diff-length-mscq-augmented-mid} $\max(b',a-1,a'+\ell'-\ell-1) + \ell'-\ell+1 > \min(b,a'+2\ell'-\ell-1)$.
			\item \label{item:diff-length-mscq-edge} If $\frac{3}{4}\ell'\le \ell<\ell'-1$, $a'-a=4\ell-3\ell'+1$ and $b'=a'+\ell'-\ell-2$, then $b\ne a+2\ell-\ell'-1$. 
			\item \label{item:diff-length-mscq-bounce-A} If $0\le a-a'\le\ell'-\ell$, $\frac{2}{3}\ell'<\ell<\frac{3}{4}\ell'$, $a+2\ell-\ell'-1\le b\le a+2(\ell'-\ell)-2$ and $b'\le a+\ell'-\ell-2$, then $b'\le a+3\ell-2\ell'-2$.
			\item \label{item:diff-length-mscq-bounce-B} If $\ell'-\ell<a-a'\le2(\ell'-\ell)$, $\frac{2}{3}\ell'<\ell<\frac{3}{4}\ell'$, $a+3\ell-2\ell'-1\le b\le a+\ell'-\ell-2$ and $b'\le a+\ell'-\ell-2$, then $b'\le a+3\ell-2\ell'-2$.
		\end{enumerate}
	\end{lemma}
	
	\begin{proof}
		(i) Suppose that $\max(a,a')-1 + \ell'-\ell \le \min(b,b')$. Let $\abar\coleq\max(a,a')$. Since $\abar\ge a$ and $\abar-1 + \ell'-\ell \le b$, we have $W[\abar,\abar+\ell'-\ell)=(W+_bc)[\abar,\abar+\ell'-\ell)=(W+_bc)[\abar+\ell,\abar+\ell')=W[\abar+\ell-1,\abar+\ell'-1)$. Since $\abar\ge a'$ and $\abar-1 + \ell'-\ell \le b'$, we have $W[\abar,\abar+\ell'-\ell)=(W+_{b'}c')[\abar,\abar+\ell'-\ell)=(W+_{b'}c')[\abar+\ell',\abar+2\ell'-\ell)=W[\abar+\ell'-1,\abar+2\ell'-\ell-1)$. But now $W[\abar+\ell-1,\abar+\ell'-1)=W[\abar+\ell'-1,\abar+2\ell'-\ell-1)$ gives a square in $W$, which is a contradiction.
		
		(ii) Suppose that $\max(b,b') + \ell'-\ell \le \min(a+\ell,a'+\ell')-2$. Let $\bbar\coleq\max(b,b')+2$. Since $\bbar\ge b+2$ and $\bbar + \ell'-\ell \le a+\ell$, we have $W[\bbar-1,\bbar+\ell'-\ell-1)=(W+_bc)[\bbar,\bbar+\ell'-\ell)=(W+_bc)[\bbar+\ell,\bbar+\ell')=W[\bbar+\ell-1,\bbar+\ell'-1)$. Since $\bbar\ge b'+2$ and $\bbar + \ell'-\ell \le a'+\ell'$, we have $W[\bbar-1,\bbar+\ell'-\ell-1)=(W+_{b'}c')[\bbar,\bbar+\ell'-\ell)=(W+_{b'}c')[\bbar+\ell',\bbar+2\ell'-\ell)=W[\bbar+\ell'-1,\bbar+2\ell'-\ell-1)$. But now $W[\bbar+\ell-1,\bbar+\ell'-1)=W[\bbar+\ell'-1,\bbar+2\ell'-\ell-1)$ gives a square in $W$, which is a contradiction.
		
		(iii) Suppose that $\max(a'-1,b) + \ell'-\ell-1 \le \min(a+\ell-2,b')$. Let $\pbar\coleq\max(a',b+1)$. Since $\pbar\ge b+1$ and $\pbar + \ell'-\ell \le a+\ell$, we have $W[\pbar,\pbar+\ell'-\ell-1)=(W+_bc)[\pbar+1,\pbar+\ell'-\ell)=(W+_bc)[\pbar+\ell+1,\pbar+\ell')=W[\pbar+\ell,\pbar+\ell'-1)$. Since $\pbar\ge a'$ and $\pbar + \ell'-\ell-2 \le b'$, we have $W[\pbar,\pbar+\ell'-\ell-1)=(W+_{b'}c')[\pbar,\pbar+\ell'-\ell-1)=(W+_{b'}c')[\pbar+\ell',\pbar+2\ell'-\ell-1)=W[\pbar+\ell'-1,\pbar+2\ell'-\ell-2)$. But now $W[\pbar+\ell,\pbar+\ell'-1)=W[\pbar+\ell'-1,\pbar+2\ell'-\ell-2)$ gives a square in $W$, which is a contradiction.
		
		(iv) Suppose that $\max(a-1,b') + \ell'-\ell+1 \le \min(a'+\ell'-2,b)$. Let $\pbar\coleq\max(a,b'+1)$. Since $\pbar\ge a$ and $\pbar + \ell'-\ell \le b$, we have $W[\pbar,\pbar+\ell'-\ell+1)=(W+_bc)[\pbar,\pbar+\ell'-\ell+1)=(W+_bc)[\pbar+\ell,\pbar+\ell'+1)=W[\pbar+\ell-1,\pbar+\ell')$. Since $\pbar\ge b'+1$ and $\pbar + \ell'-\ell \le a'+\ell'-2$, we have $W[\pbar,\pbar+\ell'-\ell+1)=(W+_{b'}c')[\pbar+1,\pbar+\ell'-\ell+2)=(W+_{b'}c')[\pbar+\ell'+1,\pbar+2\ell'-\ell+2)=W[\pbar+\ell',\pbar+2\ell'-\ell+1)$. But now $W[\pbar+\ell-1,\pbar+\ell')=W[\pbar+\ell',\pbar+2\ell'-\ell+1)$ gives a square in $W$, which is a contradiction.
		
		(v) Suppose that $\max(a-\ell'+\ell,a')-1 + \ell'-\ell \le \min(b-\ell'+\ell,b')$. Let $\abar\coleq\max(a-\ell'+\ell,a')$. Since $\abar + \ell'-\ell \ge a$ and $\abar-1 + 2(\ell'-\ell) \le b$, we have $W[\abar + \ell'-\ell,\abar + 2(\ell'-\ell))=(W+_bc)[\abar + \ell'-\ell,\abar + 2(\ell'-\ell))=(W+_bc)[\abar + \ell',\abar + 2\ell'-\ell)=W[\abar+\ell'-1,\abar+2\ell'-\ell-1)$. Since $\abar\ge a'$ and $\abar-1 + \ell'-\ell \le b'$, we have $W[\abar,\abar+\ell'-\ell)=(W+_{b'}c')[\abar,\abar+\ell'-\ell)=(W+_{b'}c')[\abar+\ell',\abar+2\ell'-\ell)=W[\abar+\ell'-1,\abar+2\ell'-\ell-1)$. But now $W[\abar,\abar+\ell'-\ell)=W[\abar+\ell'-\ell,\abar+2(\ell'-\ell))$ gives a square in $W$, which is a contradiction.
		
		(vi) Suppose that $\max(b-\ell'+\ell,b') + \ell'-\ell \le \min(a+2\ell-\ell',a'+\ell')-2$. Let $\bbar\coleq\max(b-\ell'+\ell,b')+2$. Since $\bbar + \ell'-\ell \ge b+2$ and $\bbar + 2(\ell'-\ell) \le a+\ell$, we have $W[\bbar-1 + \ell'-\ell,\bbar-1 + 2(\ell'-\ell))=(W+_bc)[\bbar + \ell'-\ell,\bbar + 2(\ell'-\ell))=(W+_bc)[\bbar + \ell',\bbar + 2\ell'-\ell)=W[\bbar-1 + \ell',\bbar-1 + 2\ell'-\ell)$. Since $\bbar\ge b'+2$ and $\bbar + \ell'-\ell \le a'+\ell'$, we have $W[\bbar-1,\bbar-1 +\ell'-\ell)=(W+_{b'}c')[\bbar,\bbar +\ell'-\ell)=(W+_{b'}c')[\bbar+\ell',\bbar +2\ell'-\ell)=W[\bbar-1+\ell',\bbar-1 +2\ell'-\ell)$. But now $W[\bbar-1,\bbar-1 +\ell'-\ell)=W[\bbar-1 + \ell'-\ell,\bbar-1 + 2(\ell'-\ell))$ gives a square in $W$, which is a contradiction.
		
		(vii) Suppose that $\max(a'-1,b-\ell'+\ell+1) + \ell'-\ell-1 \le \min(a+2\ell-\ell'-1,b')$. Let $\pbar\coleq\max(a'-1,b-\ell'+\ell+1)$. Since $\pbar+\ell'-\ell\ge b+1$ and $\pbar + 2(\ell'-\ell) \le a+\ell$, we have $W[\pbar+\ell'-\ell,\pbar+2(\ell'-\ell)-1)=(W+_bc)[\pbar+\ell'-\ell+1,\pbar+2(\ell'-\ell))=(W+_bc)[\pbar+\ell'+1,\pbar+2\ell'-\ell)=W[\pbar+\ell',\pbar+2\ell'-\ell-1)$. Since $\pbar+1\ge a'$ and $\pbar + \ell'-\ell-1 \le b'$, we have $W[\pbar+1,\pbar+\ell'-\ell)=(W+_{b'}c')[\pbar+1,\pbar+\ell'-\ell)=(W+_{b'}c')[\pbar+\ell'+1,\pbar+2\ell'-\ell)=W[\pbar+\ell',\pbar+2\ell'-\ell-1)$. But now $W[\pbar+1,\pbar+\ell'-\ell)=W[\pbar+\ell'-\ell,\pbar+2(\ell'-\ell)-1)$ gives a square in $W$, which is a contradiction.
		
		(viii) Suppose that $\max(b,b',a'+\ell'-\ell-2)+1 + \ell'-\ell \le \min(a+\ell,a'+2\ell'-\ell)-1$. Let $\pbar\coleq\max(b,b',a'+\ell'-\ell-2)+2$. Since $\pbar\ge\max(b,b')+2$ and $\pbar+\ell'-\ell\le a+\ell$, we have $(W+_{b'}c')[\pbar,\pbar+\ell'-\ell) = (W+_bc)[\pbar,\pbar+\ell'-\ell) = (W+_bc)[\pbar+\ell,\pbar+\ell')$. Since $\pbar-\ell'+\ell\ge a'$ and $\pbar\le a'+\ell'-2$, we have $(W+_bc)[\pbar+\ell,\pbar+\ell') = (W+_{b'}c')[\pbar+\ell,\pbar+\ell') = (W+_{b'}c')[\pbar-\ell'+\ell,\pbar)$. Now $(W+_{b'}c')[\pbar,\pbar+\ell'-\ell)=(W+_{b'}c')[\pbar-\ell'+\ell,\pbar)$ gives a square in $W+_{b'}c'$, so $(\pbar-\ell'+\ell,\ell'-\ell,b',c')$ is a square-completing quadruple of sign 1 in $W$. But this contradicts the minimality of the square-completing quadruple $(a',\ell',b',c')$ of sign 1 in $W$.
		 
		(ix) If $\ell=1$, then $(a,\ell,b,c)$ being of sign 1 implies $b\le a-1$ which already implies the assertion. Thus, suppose that $\ell\ge 2$ and that $\max(b',a-1,a'+\ell'-\ell-1) + \ell'-\ell+1 \le \min(b,a'+2\ell'-\ell-1)$. Let $\pbar\coleq\max(b',a-1,a'+\ell'-\ell-1)+2$. Since $\pbar\ge\max(b',a-1)+2$ and $\pbar+\ell'-\ell\le b+1$, we have $(W+_{b'}c')[\pbar,\pbar+\ell'-\ell+1) = (W+_bc)[\pbar-1,\pbar+\ell'-\ell) = (W+_bc)[\pbar+\ell-1,\pbar+\ell')$. Since $\pbar-\ell'+\ell\ge a'+1$ and $\pbar\le a'+\ell'$, we have $(W+_bc)[\pbar+\ell-1,\pbar+\ell') = (W+_{b'}c')[\pbar+\ell-1,\pbar+\ell') = (W+_{b'}c')[\pbar-\ell'+\ell-1,\pbar)$. Now $(W+_{b'}c')[\pbar,\pbar+\ell'-\ell+1)=(W+_{b'}c')[\pbar-\ell'+\ell-1,\pbar)$ gives a square in $W+_{b'}c'$, so $(\pbar-\ell'+\ell-1,\ell'-\ell+1,b',c')$ is a square-completing quadruple of sign 1 in $W$. But this, together with $\ell\ge2$, contradicts the minimality of the square-completing quadruple $(a',\ell',b',c')$ of sign 1 in $W$.
		
		(x) Suppose that $\frac{3}{4}\ell'\le \ell<\ell'-1$, $a'-a=4\ell-3\ell'+1$, $b'=a'+\ell'-\ell-2$ and $b=a+2\ell-\ell'-1$. Since $b+\ell=b'+\ell'$, $a\le b<a+\ell$ and $a'\le b'<a'+\ell'$, we have $W[b]=(W+_bc)[b]=(W+_bc)[b+\ell]=(W+_{b'}c')[b'+\ell']=(W+_{b'}c')[b']=W[b']$. Since $b+1\le a'+\ell'$, we have $W[b'+1,b)=(W+_{b'}c')[b'+2,b+1)=(W+_{b'}c')[b'+\ell'+2,b+\ell'+1)=W[b'+\ell'+1,b+\ell')$. Since $b'+\ell'-\ell\ge b$ and $b+\ell'-\ell< a+\ell$, we have $W[b'+\ell'-\ell+1,b+\ell'-\ell)=W[b+1,b+\ell'-\ell)=(W+_bc)[b'+\ell'-\ell+2,b+\ell'-\ell+1)=(W+_bc)[b'+\ell'+2,b+\ell'+1)=W[b'+\ell'+1,b+\ell')$. But now $W[b',b)=W[b,b+\ell'-\ell)$ gives a square in $W$, which is a contradiction.
		
		(xi) Suppose that $0\le a-a'\le\ell'-\ell$, $\frac{2}{3}\ell'<\ell<\frac{3}{4}\ell'$, $a+2\ell-\ell'-1\le b\le a+2(\ell'-\ell)-2$ and $a+3\ell-2\ell'-1\le b'\le a+\ell'-\ell-2$. Since $a+2(\ell'-\ell)-1\ge b+1$, we have $W[a+2(\ell'-\ell)-1,a+\ell-1)=(W+_bc)[a+2(\ell'-\ell),a+\ell)=(W+_bc)[a+2\ell'-\ell,a+2\ell)$. Since $a+\ell'-\ell-1\ge b'+1$ and $a+2\ell\le a'+2\ell'$, we have $(W+_bc)[a+2\ell'-\ell,a+2\ell) = (W+_{b'}c')[a+2\ell'-\ell,a+2\ell) = (W+_{b'}c')[a+\ell'-\ell,a+2\ell-\ell') = W[a+\ell'-\ell-1,a+2\ell-\ell'-1)$. Since $a+2\ell-\ell'-1\le b$, we have $W[a+\ell'-\ell,a+2\ell-\ell') = (W+_bc)[a+\ell'-\ell,a+2\ell-\ell') = (W+_bc)[a+\ell',a+3\ell-\ell')$. Since $a\ge a'$ and $b'\ge a+3\ell-2\ell'-1$, we have $(W+_bc)[a+\ell',a+3\ell-\ell') = (W+_{b'}c')[a+\ell',a+3\ell-\ell') = (W+_{b'}c')[a,a+3\ell-2\ell')$. Since $b\ge a+3\ell-2\ell'-1$, we have $(W+_{b'}c')[a,a+3\ell-2\ell') = (W+_bc)[a,a+3\ell-2\ell') = (W+_bc)[a+\ell,a+4\ell-2\ell')=W[a+\ell-1,a+4\ell-2\ell'-1)$. But now $W[a+2(\ell'-\ell)-1,a+\ell-1)=W[a+\ell-1,a+4\ell-2\ell'-1)$ gives a square in $W$, which is a contradiction.
		
		(xii) Suppose that $\ell'-\ell<a-a'\le2(\ell'-\ell)$, $\frac{2}{3}\ell'<\ell<\frac{3}{4}\ell'$, $a+3\ell-2\ell'-1\le b\le a+\ell'-\ell-2$ and $a+3\ell-2\ell'-1\le b'\le a+\ell'-\ell-2$. Since $a+2(\ell'-\ell)\ge b+1$, we have $W[a+2(\ell'-\ell)-1,a+\ell-1)=(W+_bc)[a+2(\ell'-\ell),a+\ell)=(W+_bc)[a+2\ell'-\ell,a+2\ell)$. Since $a+\ell'-\ell\ge b'+2$ and $a+2\ell\le a'+2\ell'$, we have $(W+_bc)[a+2\ell'-\ell,a+2\ell) = (W+_{b'}c')[a+2\ell'-\ell,a+2\ell) = (W+_{b'}c')[a+\ell'-\ell,a+2\ell-\ell')$. Since $b\le a+\ell'-\ell-2$, we have $(W+_bc)[a+\ell'-\ell,a+2\ell-\ell') = (W+_bc)[a+\ell',a+3\ell-\ell')$. Since $a\ge a'$ and $b'\ge a+3\ell-2\ell'-1$, we have $(W+_bc)[a+\ell',a+3\ell-\ell') = (W+_{b'}c')[a+\ell',a+3\ell-\ell') = (W+_{b'}c')[a,a+3\ell-2\ell')$. Since $b\ge a+3\ell-2\ell'-1$, we have $(W+_{b'}c')[a,a+3\ell-2\ell') = (W+_bc)[a,a+3\ell-2\ell') = (W+_bc)[a+\ell,a+4\ell-2\ell')=W[a+\ell-1,a+4\ell-2\ell'-1)$. But now $W[a+2(\ell'-\ell)-1,a+\ell-1)=W[a+\ell-1,a+4\ell-2\ell'-1)$ gives a square in $W$, which is a contradiction.
	\end{proof}
	
	We are now ready to prove \Cref{lem:gap-lower-bound}.
	
	\begin{proof}[Proof of \Cref{lem:gap-lower-bound}]
		Suppose that $\ell\in\N\setminus\{1\}$ and that there are distinct minimal square-completing quadruples $(a,\ell,b,c)$ and $(a',\ell,b',c')$ of sign 1 in a square-free word $W$ such that $|a-a'|\le\ell-1$. By Lemma~\ref{lem:same-length-mscq-outcomes}\ref{item:same-length-mscq-insertion-offset} we have $b\ne b'$; without loss of generality we have $b'>b$. By Lemma~\ref{lem:scq-insertion-no-right-edge} we have $b\le a+\ell-3$, and by Lemma~\ref{lem:same-length-mscq-outcomes}\ref{item:same-length-mscq-insertion-no-left-edge} we have $b'\ge a'$. Now by Lemma~\ref{lem:same-length-mscq-outcomes}\ref{item:same-length-mscq-spacing} we have $a'\ge a+\ell$, which is a contradiction. Hence, we have~$\f(\ell,\ell)\ge\ell$.
		
		Now suppose that there are distinct minimal square-completing quadruples $(a,\ell,b,c)$ and $(a',\ell',b',c')$ of sign 1 in a square-free word $W$ such that $\ell+1<\ell'\le\frac{4}{3}\ell$ and $0\le a'-a\le4\ell-3\ell'+1$. In particular, we have $\max(a,a')=a'$ and $\min(a+\ell,a'+\ell')=a+\ell$. By Lemma~\ref{lem:diff-length-mscq-outcomes}\ref{item:diff-length-mscq-left} and~\ref{item:diff-length-mscq-right} we have
		\begin{align}
			\min(b,b')&\le a'+\ell'-\ell-2\,\textrm{ and} \label{eq:diff-length-mscq-dash-left} \\
			\max(b,b')&\ge a+2\ell-\ell'-1. \, \label{eq:diff-length-mscq-dash-right}
		\end{align}
		Suppose that $b\le b'$. Since $\ell<\ell'-1$, by~\eqref{eq:diff-length-mscq-dash-right} we have $a+2\ell-\ell'-1\le \min(a+\ell-2,b')$. By Lemma~\ref{lem:diff-length-mscq-outcomes}\ref{item:diff-length-mscq-mid} we have $\min(a+\ell-2,b') < \max(a'-1,b)+\ell'-\ell-1$. Since $\ell<\ell'-1$, by~\eqref{eq:diff-length-mscq-dash-left} we have $\max(a'-1,b)+\ell'-\ell-1\le a'+2(\ell'-\ell)-3$. Putting them all together, we obtain $a+2\ell-\ell'-1<a'+2(\ell'-\ell)-3$, which implies $a'-a>4\ell-3\ell'+2$. This is a contradiction, so we have $b>b'$. Since $\ell<\ell'-1$, by~\eqref{eq:diff-length-mscq-dash-right} we have $a+2\ell-\ell'-1\le \min(a'+\ell'-2,b)$. By Lemma~\ref{lem:diff-length-mscq-outcomes}\ref{item:diff-length-mscq-mid-flip} we have $\min(a'+\ell'-2,b) \le \max(a-1,b')+\ell'-\ell$. Since $\ell<\ell'-1$ and $a\le a'$, by~\eqref{eq:diff-length-mscq-dash-left} we have $\max(a-1,b')+\ell'-\ell\le a'+2(\ell'-\ell)-2$. Since $a'-a\le4\ell-3\ell'+1$, we have $a'+2(\ell'-\ell)-2 \le a+2\ell-\ell'-1$. Putting them all together implies that all the inequalities must in fact be equalities, so we have $a'-a=4\ell-3\ell'+1$, $b'=a'+\ell'-\ell-2$ and $b=a+2\ell-\ell'-1$. But this contradicts Lemma~\ref{lem:diff-length-mscq-outcomes}\ref{item:diff-length-mscq-edge}, so~$\f(\ell,\ell') \ge 4\ell-3\ell'+2$ for all~$\ell'\in\N\sm\{1\}$ satisfying~$\ell+1<\ell'\le\frac{4}{3}\ell$.
		
		Suppose that there are distinct minimal square-completing quadruples $(a,\ell,b,c)$ and $(a,\ell',b',c')$ of sign 1 in a square-free word $W$ such that $\frac{4}{3}\ell<\ell'<\frac{3}{2}\ell$. By applying Lemmas~\ref{lem:diff-length-mscq-outcomes}\ref{item:diff-length-mscq-left}, \ref{item:diff-length-mscq-right} and~\ref{item:diff-length-mscq-mid-flip} we have, respectively,
		\begin{align}
			\min(b,b')&\le a+\ell'-\ell-2\,, \label{eq:diff-length-mscq-same-left} \\
			\max(b,b')&\ge a+2\ell-\ell'-1\,\textrm{ and} \label{eq:diff-length-mscq-same-right} \\
			\min(a+\ell'-2,b) &< \max(a-1,b') + \ell'-\ell+1\,.  \label{eq:diff-length-mscq-same-mid-flip}
		\end{align}
		Suppose that $b\le b'$. By~\eqref{eq:diff-length-mscq-same-right} we have $a+2\ell-\ell'-1 = \min(a+2\ell-\ell'-1,b')$ and by~\eqref{eq:diff-length-mscq-same-left} we have $\max(a+\ell'-\ell-2,b)= a+\ell'-\ell-2$, so by applying Lemma~\ref{lem:diff-length-mscq-outcomes}\ref{item:diff-length-mscq-shifted-mid} and simplifying, we obtain $3\ell-2\ell'+1<0$. This is a contradiction, so we have $b>b'$. Since $3\ell-2\ell'>0$ and $b\le a+\ell-2\le a+\ell'-2$, by~\eqref{eq:diff-length-mscq-same-right} and~\eqref{eq:diff-length-mscq-same-mid-flip} we have 
		\[a+\ell'-\ell\le a+2\ell-\ell'-1\le b=\min(a+\ell'-2,b)<\max(a-1,b') + \ell'-\ell+1\,.\]
		This implies $a\le\max(a-1,b')$, so $\max(a-1,b')=b'$. Combining further with~\eqref{eq:diff-length-mscq-same-left} yields
		\[a+2\ell-\ell'-1\le b\le b'+\ell'-\ell\le a+2(\ell'-\ell)-2.\]
		Hence, we obtain $4\ell-3\ell'<0$, $b\le a+2(\ell'-\ell)-2$ and $b'\ge a+3\ell-2\ell'-1$. But this contradicts Lemma~\ref{lem:diff-length-mscq-outcomes}\ref{item:diff-length-mscq-bounce-A}, so~$\f(\ell,\ell') \ge 1$ for all~$\ell'\in\N\sm\{1\}$ satisfying~$\frac{4}{3}\ell<\ell'<\frac{3}{2}\ell$.
		
		Suppose that we have distinct minimal square-completing quadruples $(a,\ell,b,c)$ and $(a',\ell',b',c')$ of sign 1 such that $\frac{2}{3}\ell<\ell'<\ell-1$ and $0\le a'-a<\ell'$. In particular, we have $\max(a,a')=a'$. By Lemmas~\ref{lem:diff-length-mscq-outcomes}\ref{item:diff-length-mscq-left}, \ref{item:diff-length-mscq-right} and~\ref{item:diff-length-mscq-mid-flip} we have, respectively,
		\begin{align}
			\min(b,b')&\le a'+\ell-\ell'-2\,, \label{eq:diff-length-mscq-left} \\
			\max(b,b')&\ge \min(a+\ell,a'+\ell')+\ell'-\ell-1\,\textrm{ and} \label{eq:diff-length-mscq-right} \\
			\min(a+\ell-2,b') &< \max(a'-1,b) + \ell-\ell'+1\,.  \label{eq:diff-length-mscq-mid-flip}
		\end{align}
		
		Suppose that $b'\le b$. By Lemma~\ref{lem:diff-length-mscq-outcomes}\ref{item:diff-length-mscq-shifted-mid} we have 
		\begin{equation} \label{eq:diff-length-mscq-shifted-mid}
			\min(a'+2\ell'-\ell-1,b) < \max(a+\ell-\ell'-2,b')\,.
		\end{equation}
		Suppose further that $a'+\ell'\le a+\ell$. By~\eqref{eq:diff-length-mscq-right} we have $a'+2\ell'-\ell-1 = \min(a'+2\ell'-\ell-1,b)$. Since $a\le a'$, by~\eqref{eq:diff-length-mscq-left} we have $\max(a+\ell-\ell'-2,b')\le a'+\ell-\ell'-2$. Putting both inequalities together with~\eqref{eq:diff-length-mscq-shifted-mid} and simplifying, we obtain $3\ell'-2\ell+1<0$. This is a contradiction, so we have $a'+\ell'>a+\ell$. Since $a<a'-\ell+\ell'\le b'-\ell+\ell'+1$, we have $\max(a+\ell-\ell'-2,b')=b'$. Since $3\ell'-2\ell-1\ge0$, by~\eqref{eq:diff-length-mscq-left} we have $b'\le b'+3\ell'-2\ell+1\le a'+2\ell'-\ell-1$. Putting both inequalities together with~\eqref{eq:diff-length-mscq-shifted-mid} and $b'\le b$, we obtain
		\[ b' \le \min(a'+2\ell'-\ell-1,b) < \max(a+\ell-\ell'-2,b') = b'.\]
		This is a contradiction, so we have $b'>b$.
		
		Suppose that $a'+\ell'\le a+\ell$. Since $3\ell'-2\ell>0$ and $b'\le a'+\ell'-2\le a+\ell-2$, by~\eqref{eq:diff-length-mscq-right} we have $a'+\ell-\ell'\le a'+2\ell'-\ell-1\le b'=\min(a+\ell-2,b')$. Putting this together with~\eqref{eq:diff-length-mscq-mid-flip} yields $a'\le\max(a'-1,b)$, so $\max(a'-1,b)=b$. Combining further with~\eqref{eq:diff-length-mscq-left}, we obtain
		\[a'+2\ell'-\ell-1\le b'\le b+\ell-\ell'\le a'+2(\ell-\ell')-2.\]
		Hence, we obtain $4\ell'-3\ell<0$, $b'\le a'+2(\ell-\ell')-2$ and $b\ge a'+3\ell'-2\ell-1$. But this contradicts Lemma~\ref{lem:diff-length-mscq-outcomes}\ref{item:diff-length-mscq-bounce-A}, so we have $a'+\ell'>a+\ell$. By Lemma~\ref{lem:diff-length-mscq-outcomes}\ref{item:diff-length-mscq-shifted-left} we have
		\begin{equation} \label{eq:diff-length-mscq-shifted-left}
			a'-1 > \min(b'-\ell+\ell',b)\,,
		\end{equation}
		and by Lemma~\ref{lem:diff-length-mscq-outcomes}\ref{item:diff-length-mscq-shifted-right} we have
		\begin{equation} \label{eq:diff-length-mscq-shifted-right}
			\max(b',b+\ell-\ell') > \min(a'+2\ell'-\ell,a+\ell)-2\,.
		\end{equation}
		
		Suppose that $b'\ge a+\ell-2$, and suppose further that $b\le a'-1$. By~\eqref{eq:diff-length-mscq-mid-flip} we have $a'-a=\ell'-1$. Since $3\ell'-2\ell-1\ge0$, we have $a'+2\ell'-\ell\ge a'+\ell-\ell'+1 = a+\ell$. Furthermore, we have $b'\ge a+\ell-2=a'+\ell-\ell'-1\ge b+\ell-\ell'$. Hence, by~\eqref{eq:diff-length-mscq-shifted-right} we have $b'\ge a+\ell-1$. Now since $\ell-2\ell'+1\le0$ and by~\eqref{eq:diff-length-mscq-mid-flip} we have $a'-a=\ell'-1$, we have $\max(a'-1,b,a+\ell-\ell'-1) + \ell-\ell'+1 = a'+\ell-\ell' = a+\ell-1$. Since $\ell'<\ell-1$ and $b'\ge a+\ell-1$, we have $\min(b',a+2\ell-\ell'-1) \ge a+\ell-1$. By Lemma~\ref{lem:diff-length-mscq-outcomes}\ref{item:diff-length-mscq-augmented-mid} we have
		\[a+\ell-1=\max(b,a'-1,a+\ell-\ell'-1) + \ell-\ell'+1 > \min(b',a+2\ell-\ell'-1)\ge a+\ell-1,\]
		which gives a contradiction.
		
		Hence, we have $b\ge a'$. By~\eqref{eq:diff-length-mscq-shifted-left} we have $\min(b'-\ell+\ell',b)\le a'-2$, which implies $b'-\ell+\ell'\le a'-2$. Putting this together with $b'\ge a+\ell-2$ and $a'-a\le\ell'-1$, we obtain $a+\ell-2\le b'\le a'+\ell-\ell'-2\le a+\ell-3$. This gives a contradiction, so $b'\le a+\ell-3$.
		
		Suppose that $b\le a'-1$, and suppose further that $a'-a\le2(\ell-\ell')$. In particular, we have $\min(a'+2\ell'-\ell,a+\ell)=a'+2\ell'-\ell$. By~\eqref{eq:diff-length-mscq-shifted-right} we have $a'+2\ell'-\ell-1\le \max(b',b+\ell-\ell')$, and by~\eqref{eq:diff-length-mscq-mid-flip} we have $b'\le a'+\ell-\ell'-1$. Since $b\ge a-1$, we obtain
		\[a'+2\ell'-\ell-1\le\max(b',b+\ell-\ell')\le\max(a,a')+\ell-\ell'-1\le a'+2\ell'-\ell-2\,,\]
		which gives a contradiction. Hence, $a'-a>2(\ell-\ell')$. Now by~\eqref{eq:diff-length-mscq-shifted-right} we have
		\[\max(b',b+\ell-\ell')-1\ge a+\ell-2\ge b',\]
		so $a'+\ell-\ell'-2\ge b+\ell-\ell'-1\ge a+\ell-2$. This implies $a'-a\ge\ell'$, which gives a contradiction. Hence, we have $b\ge a'$.
		
		Since by~\eqref{eq:diff-length-mscq-mid-flip} we have $b+\ell-\ell'\ge b'$, by~\eqref{eq:diff-length-mscq-shifted-left} we have $b'\le a'+\ell-\ell'-2$ and by~\eqref{eq:diff-length-mscq-shifted-right} we have $b\ge\min(a'+3\ell'-2\ell,a+\ell')-1$. Suppose that $a'-a\le2(\ell-\ell')$. In particular, we have $\min(a'+3\ell'-2\ell,a+\ell')=a'+3\ell'-2\ell$. By~\eqref{eq:diff-length-mscq-left} we have
		\[a'+3\ell'-2\ell-1\le b\le a'+\ell-\ell'-2,\]
		so we have $4\ell'-3\ell<0$. But this contradicts Lemma~\ref{lem:diff-length-mscq-outcomes}\ref{item:diff-length-mscq-bounce-B}, so we have $a'-a>2(\ell-\ell')$. In particular, we have $\min(a'+3\ell'-2\ell,a+\ell')=a+\ell'$. By~\eqref{eq:diff-length-mscq-left} we have
		\[a+\ell'-1\le b\le a'+\ell-\ell'-2,\]
		so we have $a'-a>2\ell'-\ell$. Now since $a'-a>2(\ell-\ell')$, we have $\min(a'+\ell',a+2\ell-\ell')=a+2\ell-\ell'$; since $3\ell'-2\ell,\ell-\ell'>0$, we have $a+\ell-\ell'-2\le a+\ell'-1\le b<b'$, which gives $\max(b,b',a+\ell-\ell'-2)=b'$. Hence, by Lemma~\ref{lem:diff-length-mscq-outcomes}\ref{item:diff-length-mscq-augmented-right} we have $a+2\ell-\ell'=\min(a'+\ell',a+2\ell-\ell')<\max(b,b',a+\ell-\ell'-2)+2+\ell-\ell' = b'+2+\ell-\ell' \le a'+2(\ell-\ell')$, which implies $a'-a>\ell'$. But this is a contradiction, so~$\f(\ell,\ell') \ge \ell'$ for all~$\ell'\in\N\sm\{1\}$ satisfying~$\frac{2}{3}\ell<\ell'<\ell-1$.
		
		Finally, note that~$\f(\ell,\ell') \ge 0$ for all~$\ell,\ell'\in\N\sm\{1\}$.
	\end{proof}
	
	\section{Concluding remarks}
	
	In this paper, we proved that there are no extremal square-free words over alphabets of size at least~$5$, thereby confirming Conjecture~\ref{conj:no-extremal-squarefree-4} for all~$k\ge5$. The only case which remains open is that for~$k=4$, which corresponds to the original conjecture of Grytczuk, Kordulewski and Niewiadomski in~\cite{GrytczukKordulewskiNiewiadomski}.
	
	\begin{conjecture}[{\cite{GrytczukKordulewskiNiewiadomski}}] \label{conj:no-extremal-squarefree-4-only}
		There is no extremal square-free word over an alphabet of size~$4$.
	\end{conjecture}
	
	Broadly speaking, our proof works for~$k\ge5$ because the coefficient of~$n$ in~\Cref{eq:mscq-collection-bounds} is strictly and meaningfully less than~$5$. A first approach to tackle the case~$k=4$ would be to try to obtain a smaller coefficient which is strictly less than~$4$ computationally, but it is not entirely clear whether such a substantial gain would be possible.
	
	\section*{Acknowledgements}
	
	Eng Keat Hng is supported by the Institute for Basic Science (IBS-R029-C4). Silas Rathke is funded by the Deutsche Forschungsgemeinschaft (DFG, German Research Foundation) under Germany's Excellence Strategy – The Berlin Mathematics Research Center MATH+ (EXC-2046/1, project ID: 390685689). We thank the HPC service of the Freie Universit\"at Berlin for the computation time provided (10.17169/refubium-26754). We thank David Lecumberri Irisarri, Katherine Perry, and Bjarne Schülke for stimulating discussions at the early stages of this project. We are grateful for the support of the 1st Early Career Workshop in Extremal Combinatorics 2025 organized at ECOPRO, IBS.
	
	\bibliographystyle{amsplain}
	\bibliography{bib}
	
\end{document}